\documentclass[preprint,10pt]{elsarticle} 
\makeatletter
    \def\ps@pprintTitle{%
  \def\@oddfoot{\reset@font\hfil\thepage\hfil\textit{\today}}
                       }
\makeatother
\usepackage[T1,IL2]{fontenc}
\usepackage[english]{babel}
\usepackage[cp1250]{inputenc}
\usepackage{amsthm}%\begin{proof}...\end{proof}
\usepackage{amsmath, amssymb}
\usepackage{enumerate}
\usepackage{hyperref}
\usepackage{upgreek}%velke grecke pismena, napr. \upalpha
\usepackage{color}
\usepackage{longtable}%Make a table that takes up more than a single page?
\usepackage{array}%Matrices
\usepackage{multirow}%In tabular enviroment
\newtheorem{theorem}{Theorem}
\newtheorem{lemma}{Lemma}

\newtheorem{corollary}{Corollary}
\theoremstyle{definition}

\theoremstyle{remark}

\def\emphasize#1{\textit{#1}}
\def\mmod#1{\!\!\!\!\pmod{#1}}
\newcommand{\Z}{\ensuremath{{\mathcal Z}}}
\newcommand{\zn}{\ensuremath{\Z_n}}
\newcommand{\zk}{\ensuremath{\Z_k}}

\newcommand{\deter}[4]{\left\vert\begin{array}{rr} #1 & #2\\ #3 & #4 \end{array}\right\vert}

\newcommand{\mmatrix}[4]{\left(\begin{array}{rr} #1 & #2\\ #3 & #4 \end{array}\right)}

\newcommand{\ax}{\ensuremath{(x,x;5,0,0)}}\newcommand{\ay}{\ensuremath{(y,y;5,0,0)}}
\newcommand{\axm}{\ensuremath{(x,-x;5,0,0)}}\newcommand{\aym}{\ensuremath{(y,-y;5,0,0)}}
\newcommand{\bx}{\ensuremath{(x,-x;0,0,1)}}\newcommand{\by}{\ensuremath{(y,-y;0,0,1)}}
\newcommand{\cx}{\ensuremath{(x,x;1,0,1)}}\newcommand{\cy}{\ensuremath{(y,y;1,0,1)}}
\newcommand{\cxm}{\ensuremath{(-x,\alpha^4 x;0,9,1)}}\newcommand{\cym}{\ensuremath{(-y,\alpha^4 y;0,9,1)}}
\newcommand{\dx}{\ensuremath{(x,0;5,0,1)}}\newcommand{\dy}{\ensuremath{(y,0;5,0,1)}}
\newcommand{\dxm}{\ensuremath{(0,x;0,5,1)}}\newcommand{\dym}{\ensuremath{(0,y;0,5,1)}}
\newcommand{\ex}{\ensuremath{(x,\alpha^3 x;1,3,1)}}\newcommand{\ey}{\ensuremath{(y,\alpha^3 y;1,3,1)}}
\newcommand{\exm}{\ensuremath{(-x,\alpha^4 x;7,9,1)}}\newcommand{\eym}{\ensuremath{(-y,\alpha^4 y;7,9,1)}}
\newcommand{\fx}{\ensuremath{(x,x;1,7,1)}}\newcommand{\fy}{\ensuremath{(y,y;1,7,1)}}
\newcommand{\fxm}{\ensuremath{(-\alpha^3 x,\alpha^4 x;3,9,1)}}\newcommand{\fym}{\ensuremath{(-\alpha^3 y,\alpha^4 y;3,9,1)}}
\newcommand{\gx}{\ensuremath{(0,x;5,2,1)}}\newcommand{\gy}{\ensuremath{(0,y;5,2,1)}}
\newcommand{\gxm}{\ensuremath{(\alpha^3 x,0;8,5,1)}}\newcommand{\gym}{\ensuremath{(\alpha^3 y,0;8,5,1)}}
\newcommand{\hx}{\ensuremath{(x,\alpha x;3,2,0)}}\newcommand{\hy}{\ensuremath{(y,\alpha y;3,2,0)}}
\newcommand{\hxm}{\ensuremath{(\alpha^2 x,\alpha^4 x;7,8,0)}}\newcommand{\hym}{\ensuremath{(\alpha^2 y,\alpha^4 y;7,8,0)}}
\newcommand{\kx}{\ensuremath{(x,\alpha x;4,1,0)}}\newcommand{\ky}{\ensuremath{(y,\alpha y;4,1,0)}}
\newcommand{\kxm}{\ensuremath{(\alpha x,\alpha^4 x;6,9,0)}}\newcommand{\kym}{\ensuremath{(\alpha y,\alpha^4 y;6,9,0)}}
\newcommand{\Group}{\thegroup. \addtocounter{group}{1}}
\journal{European Journal of Combinatorics}
\begin{document} 
\def\xa{20800} \def\xb{21221} \def\xc{4.2\cdot 10^6} \def\xd{10^{10}} \def\eps{0.002785} \def\delt{0.0055856} \def\deltA{1.0055856} \def\gam{1.412480}
\def\siran{\v Sir\'a\v n}\def\siagi{\v Siagiov\'a}\def\zdima{\v Zd\'\i malov\'a}
\def\koef{0.684} \def\xasharp{20801} \def\dsharp{353616} \def\dsharpPrime{360756}
\hyphenation{ge-ne-ra-tors ge-ne-ra-ting ge-ne-ra-ted}
\newcounter{group}
%%%%%%%%%%%%%%%%%%%%%%%%%%%%%%%%%%%%%%%%%%%%%%%%%%%%%%%%%%%%%%%%%%%%%%%%%%%%%%%%%%%%%%%%
%%%%%%%%%%%%%%%%%%%%%%%%%%%%%%%%%%%%%%%%%%%%%%%%%%%%%%%%%%%%%%%%%%%%%%%%%%%%%%%%%%%%%%%%
\begin{frontmatter}

\title{Cayley graphs of diameter two with order greater than \koef\ of the Moore bound for any degree}

\author[rvt]{Marcel Abas\corref{cor1}}
\ead{abas@stuba.sk}
\cortext[cor1]{Corresponding author}

\address[rvt]{Institute of Applied Informatics, Automation and Mathematics,\\
Faculty of Materials Science and Technology in Trnava,\\
Slovak University of Technology in Bratislava, Trnava, Slovakia}
%Phone: 00421 918 646021, Fax:   00421 906 068299}

\begin{abstract}\label{abstract} 
It is known that the number of vertices of a graph of diameter two cannot exceed $d^2+1$. In this contribution we give a new lower bound for orders of Cayley graphs of 
diameter two in the form 
$C(d,2)>\koef d^2$ valid for all degrees $d\geq\dsharpPrime$. The result is a significant improvement of currently known results on the orders of Cayley graphs of diameter two.
\end{abstract}

\begin{keyword}
Degree; Diameter; Moore bound; Cayley graph, Networks.
\end{keyword}

\end{frontmatter}
%%%%%%%%%%%%%%%%%%%%%%%%%%%%%%%%%%%%%%%%%%%%%%%%%%%%%%%%%%%%%%%%%%%%%%%%%%%%%%%%%%%%%%%%%%%%%%%%%%%%%%%%%%%%
%%%%%%%%%%%%%%%%%%%%%%%%%%%%%%%%%%%%%%%%%%%%%%%%%%%%%%%%%%%%%%%%%%%%%%%%%%%%%%%%%%%%%%%%%%%%%%%%%%%%%%%%%%%%
\section{Introduction}\label{sec_introduction}

Networks (optical, interconnection, etc.) are usually modeled by graphs (digraphs), where nodes of the network are represented by vertices and communication lines 
are represented by undirected (directed) edges. Therefore the designing of networks is closely linked to constructing (di)graphs with some preassigned properties. Obviously, the basic limitation on any network is the number of communication lines connected to a node and the maximum communication delay. These two parameters correspond to the degree 
(in/out degree) of the corresponding vertex and to the diameter of the (di)graph, respectively.

From now on, by graph we will mean a graph, which is finite, connected and simple (i.e. undirected, without loops and multiple edges).
In graph theory, the problem to find the largest order $n(d,k)$ of a graph with given maximum degree $d$ and diameter $k$ is known as \emphasize{degree--diameter} problem. The well known upper bound on the number $n(d,k)$ is \emphasize{Moore bound} which gives  
$n(d,k)\leq 1+d+d(d-1)+d(d-1)^2+\cdots+d(d-1)^{k-1}$ for every $d, k\geq1$. The graphs satisfying the Moore bound are known as \emphasize{Moore graphs}. Except $k=1$ or $d\leq2$ there are only few graphs achieving the Moore bound. For $k=1,d\geq 1$ we have $n(d,1)=d+1$ (achieved by complete graphs $K_{d+1}$) and for $d=2,k\geq 1$ we have
$n(2,d)=2d+1$ (odd-length cycles $C_{2k+1}$). In \cite{HS1960} Hoffman and Singleton proved that for $k=2$ there exists Moore graph only for degrees $d\in\{2,3,7\}$ and (maybe)
for $d=57$. For the first three values of $d$ the corresponding Moore graph is unique. It is the 5-cycle $C_5$ for $d=2$, Petersen graph for $d=3$ and Hoffman--Singleton graph for 
$d=7$. Until now it is not known yet if a Moore graph of degree $d=57$ and diameter $k=2$ on 3250 vertices does exist. However, it was shown by Cameron \cite{C1999} that if there is a Moore graph for $d=57$ and $k=2$, the graph is not vertex-transitive, in contrast to the fact that the other Moore graphs are vertex-transitive. 
In addition, Ma\v caj and \v Sir\'a\v n proved \cite{MS2010} that the order of the automorphism group of the (hypothetical) Moore graph is at most 375. 
It was shown by Damerell \cite{D1973} that for $d\geq 3$ and $k\geq 3$ there is no Moore graph. That is for other combinations of degrees and diameters there are no other Moore graphs.
For a summary on the history and development on this topic we refer to survey paper of Miller and \v Sir\'a\v n \cite{MS2005}.

The Moore bound for diameter two is $n(d,2)\leq d^2+1$, and it was shown by Erd\"os, Fajtlowicz and Hoffman \cite{EFH1980} that for $d\geq4$, $d\ne7$ and $d\ne 57$, the bound 
is $n(d,2)\leq d^2-1$. An explicit lower bound $n(d,2)\geq d^2-d+i$ is given by Brown's graphs \cite{B1996} for all $d$ such that $d-1$ is a prime power and $i=2$ for $d-1$ even and $i=1$ for $d-1$ odd. A modification of Brown's graphs constructed by  \siran, \siagi\ and \zdima\ \cite{SSZ2010} gives the lower bound $n(d,2)\geq d^2-2d^{1.525}$ for all sufficiently large $d$.
Since neither the Brown's graphs nor their modification are vertex-transitive, it is a natural question to ask what is the maximum number of vertices  
of a vertex-transitive graph or a Cayley graph of diameter two and degree $d$. These numbers we will denote by $v(d,2)$ and $C(d,2)$, respectively. As the Cayley graphs are 
vertex-transitive, we have $n(d,2)\geq v(d,2)\geq C(d,2)$ for each degree $d$.

Currently, the best known construction of vertex-transitive graphs are {McKay-Miller-\v Sir\'a\v n} graphs \cite{MMS1998} which give $v(d,2)\geq\frac89(d+\frac12)^2$, for degrees 
$d=\frac12(3q-1)$ such that $q\equiv 1\pmod{4}$ is a prime power. In the same paper the authors have shown that all these graphs are non-Cayley.
%In \cite{S2001} \siagi\ has shown that all these graphs are non-Cayley.

For Cayley graphs we have the following results. Until recently, the best known lower bound valid for all degrees $d$ there was a folklore bound giving
$C(d,2)\geq \frac14 d^2+d+1$ for even $d$ and $\frac14 d^2+d+\frac34$ for odd $d$. \siagi\ and \siran\ in \cite{SS2005} constructed Cayley graphs of diameter two and of order $\frac12(d+1)^2$ for all degrees $d=2q-1$ where $q$ is an odd prime power and the same authors gave a construction \cite{SS2012} of Cayley graphs of diameter two and of order $d^2-O(d^{\frac32})$ for an infinite set of degrees $d$ of a very special type. Very recently Abas \cite{A2014, A2015} has shown that for all degrees $d\geq4$ we have $C(d,2)\geq\frac12 d^2$. 

In this contribution we show that for every degree $d=17n-1$, $n\geq 1$, such that $n\equiv 1\pmod{10}$ is a prime, there is a Cayley graph of diameter two and of order $\frac{200}{289}(d+1)^2$. Using explicit estimates for the distribution of primes in arithmetic progressions
we show that for every $d\geq\dsharpPrime$ there is a Cayley graph of diameter two and of order greater than $\koef d^2$. This is a significant improvement of the known results valid for all degrees. In addition, we show that our construction provides infinitely many currently largest known Cayley graphs of diameter two.
%%%%%%%%%%%%%%%%%%%%%%%%%%%%%%%%%%%%%%%%%%%%%%%%%%%%%%%%%%%%%%%%%%%%%%%%%%%%%%%%%%%%%%%%%%%%%%%%%%%%%%%%%%%%
%%%%%%%%%%%%%%%%%%%%%%%%%%%%%%%%%%%%%%%%%%%%%%%%%%%%%%%%%%%%%%%%%%%%%%%%%%%%%%%%%%%%%%%%%%%%%%%%%%%%%%%%%%%%
\section{Preliminaries}\label{sec_preliminaries}

Let $\Gamma$ be a finite group and let $X\subset\Gamma$ be a unit-free, inverse-closed generating set for $\Gamma$. That is, $1_\Gamma\not\in X$, $X=X^{-1}$ and 
$\Gamma=\langle X\rangle$. 
The Cayley graph $G=Cay(\Gamma,X)$ for the \emphasize{underlying} group $\Gamma$ and the \emphasize{generating} set $X$ is a graph with vertex set 
$V(G)=\Gamma$ and edge set $E(G)=\{\{g,h\}\vert g\in\Gamma, g^{-1}h\in X\}$. Since the set $X$ is inverse closed, $g^{-1}h\in X$ imply $h^{-1}g\in X$ and therefore our Cayley graphs are undirected. Note that the mapping $\varphi_h:V(G)\to V(G)$ defined by $\varphi_h(g)=hg,\ g\in V(G)=\Gamma$, is a graph automorphism for each $h\in \Gamma$. It follows that every Cayley graph is vertex-transitive.

Throughout this paper, we will denote by $\zn$ the cyclic group of order $n$, by $\zn^2=\zn\times\zn$ the direct product of $\zn$ with itself, and by $\zn^{\times}$ the multiplicative group of units modulo $n$. We will write the cyclic group $\zn$ as an additive group with elements $0,1,\dots,n-1$, with identity element $0$ and with $-x$ as the inverse element to $x$. The elements of $\zn^2$ will be written in the form $(x,y),\ x,y\in\zn$, and the product of two elements of $\zn^2$ will be $(x_1,y_1)\cdot(x_2,y_2)=(x_1+x_2,y_1+y_2)$. 

Let $\alpha\in\zn^{\times}$ and let $A:\zn^2\to\zn^2$ be an automorphism of $\zn^2$ given by $A: (x,y)\to(\alpha x,y)$. Clearly, if the order of $\alpha$ in $\zn^{\times}$ is $k$ then $k$ is the order of the automorphism $A$.
Let $B:\zn^2\to\zn^2$ be another automorphism of $\zn^2$ given by $B: (x,y)\to(y,x)$ and let $\Delta=\Delta(n,\alpha)$ be the group generated by $A$ and $B$. That is, $\Delta(n,\alpha)=\langle A,B\rangle$. 
One can see that if the order of $\alpha\in\zn^{\times}$ is the same as the order of $\alpha'\in{\mathcal Z}_{n'}^{\times}$, say $k$, then the groups $\Delta(n,\alpha)$ and $\Delta(n',\alpha')$ are isomorphic.
Therefore the structure of $\Delta(n,\alpha)$ depends only on the order $k$ of $\alpha$ in $\zn^{\times}$ and we will write it simply as $\Delta_k$. The group $\Delta_k$ has a presentation $\Delta_k=\langle A,B\vert A^k=B^2=(AB)^2(A^{-1}B)^2=1_{\Delta_k}\rangle$. 

Now suppose we have a group with presentation $\langle a,b,c\vert a^k=b^k=c^2, ba=ab, cb=ac\rangle$. One can see that this group is isomorphic to the group $\Delta_k$ via 
the mapping $A\to a, BAB\to b$ and $B\to c$. Elements of the group $\langle a,b,c\rangle$ can be written in the form $(x,y,i)$, where $x,y\in\zk$ and $i\in\Z_2$. Therefore the group $\Delta_k$ is isomorphic to semidirect product of $\zk^2$ with $\Z_2$, that is $\Delta_k=\zk^2\rtimes \Z_2$, where the non-identity element of $\Z_2$ interchanges the coordinates of elements of $\zk^2$. We will write the elements of $\Delta_k$ as triples $(x,y,i)$, where $x,y\in\{0,1,\dots,k-1\}$ and $i\in\{0,1\}$. The product of two elements of 
$\Delta_k$ is given by $(x_0,x_1,i)\cdot(y_0,y_1,j)=(x_0+y_i,x_1+y_{1-i},i+j)$, where the first two coordinates are taken modulo $k$ and the last coordinate is taken modulo 2.
The inverse element to $(x_0,x_1,i)$ is $(-x_i,-x_{1-i},i)$. Clearly, the order of the group $\Delta_k$ is $2k^2$.

Let $\alpha\in\zn^\times$ have order $k$. We will denote the semidirect product $\zn^2\rtimes_\alpha\Delta_k$ by the symbol  
$\Gamma(n,\alpha)$ (obviously, $\Gamma$ is uniquely determined by $n$ and $\alpha$). For the automorphisms $a,b,c\in\Delta_k$ of $\zn^2$ then we have:
$a:(x,y)\to (\alpha x,y)$, $b:(x,y)\to (x,\alpha y)$ and $c:(x,y)\to (y,x)$. It is easy to see that the order of the group $\Gamma(n,\alpha)$ is $2n^2k^2$. We will write the elements of $\Gamma(n,\alpha)$ as quintuples $(x_0,x_1;y_0,y_1,i),\ x_0,x_1\in\zn,\ y_0,y_1\in\zk,\ i\in\Z_2$. Since the automorphism $(y_0,y_1,i)\in\Delta_k$ maps an element $(x_0,x_1)$ of $\zn^2$ to the element $(\alpha^{y_0}x_i,\alpha^{y_1}x_{1-i})$, for the product of two elements of $\Gamma(n,\alpha)$ we have 
\scalebox{.9}{$(x_0,x_1;y_0,y_1,i)\cdot(x'_0,x'_1;y'_0,y'_1,i')=(x_0+\alpha^{y_0}x'_{i},x_1+\alpha^{y_1}x'_{1-i};y_0+y'_i,y_1+y'_{1-i},i+i')$},\\ 
where the first two coordinates are taken modulo $n$, the second two are taken modulo $k$ and the last coordinate is taken modulo 2.

In the following two lemmas we show that prime numbers of the form $p\equiv 1\pmod{10}$ have some special
properties. We will need the lemmas in the proof of Theorem \ref{th_group_200}.
 
\begin{lemma}\label{lem_alpha_p}
Let $p=10s+1$, $s\geq1$, be a prime number. Then there exists an integer $\alpha$, $1<\alpha<p-1$, such that:\\

\noindent
i)\phantom{ii} $\alpha^{10}\equiv 1\pmod{p}$,\\ 
ii)\phantom{i} $gcd(\alpha^i-1,p)=1$, $2\leq i\leq 5$,\\
iii) $gcd(\alpha^i+1,p)=1$, $i=3,4$,\\
iv)\phantom{i} $\alpha^5\equiv -1\pmod{p}$.
\end{lemma}

\begin{proof}
Let $a$ be a primitive root of $p$. The multiplicative order of $a$ modulo $p$ is equal to $\varphi(p)=10s$ 
(note that $\varphi(n)$ is the Euler's totient function -  the number of positive integers less than or equal to $n$ that are relatively prime to $n$). 
Let $\alpha=a^{s}$. Since $\alpha^2=a^{2s}\not\equiv 1\pmod{p}$, or equivalently $gcd(\alpha^2-1,p)=1$, we have 
$\alpha\not\equiv\pm1\pmod{p}$ and consequently $1<\alpha<p-1$. We claim that $\alpha$ has the required properties i), ii), iii) and iv).\\

\noindent
i) $\alpha^{10}\equiv 1\pmod{p}$: Clearly, $\alpha^{10}=a^{10s}=a^{\varphi(p)}\equiv 1\pmod{p}$.\\
ii) $gcd(\alpha^i-1,p)=1$, $2\leq i\leq 5$: It follows immediately from the fact that $\alpha$ is a power of a primitive root of $p$.\\
iii) $gcd(\alpha^i+1,p)=1$, $i=3,4$: If $gcd(\alpha^i+1,p)>1$ then $\alpha^i\equiv -1\pmod p$ and consequently $\alpha^{2i}\equiv 1\pmod p$, for $2i=6$ or $8$, which is a contradiction with the definition of $\alpha$.\\
iv) $\alpha^5\equiv -1\pmod{p^k}$: From $\alpha^{10}\equiv 1\pmod{p^k}$ we have $\alpha^{10}-1=(\alpha^5+1)(\alpha^5-1)\equiv 0\pmod{p^k}$. 
Since $\alpha^5\not\equiv 1\pmod{p}$, and $p$ is a prime, it follows that $\alpha^5\equiv -1\pmod{p}$.
\end{proof}

\begin{lemma}\label{lem_Lambda}
Let $p$ and $\alpha$ be as in the previous lemma and
let $\Lambda=\Lambda_1\cup\Lambda_2$, where $\Lambda_1=\{\pm\alpha^i,\pm2\alpha^i\vert 0\leq i\leq 4\}$ and 
$\Lambda_2=\{\pm\alpha^i\pm\alpha^j\vert i,j=0,1,2,3,4, i\ne j\}$. Then every element of $\Lambda$ is coprime with $p$.
\end{lemma}

\begin{proof}$ $\\
1) Let $\lambda\in\Lambda_1$. Since $2,-2$ and $\alpha^i,-\alpha^i$ are coprimes with $p$, the result follows.\\ 
2) Let $\lambda\in\Lambda_2$ and let $j=0$. That is $\lambda=\pm\alpha^i\pm1$. It follows from the previous lemma that $gcd(\lambda,p)=1$.\\
3) Let $i,j>0$ and let (without loss of generality) $j\leq i$. Then $\lambda=\pm\alpha^i\pm\alpha^j=\alpha^j\cdot(\pm\alpha^{i'}\pm1)$, where $0\leq i'\leq 3$.
Since both factors of the product are coprime with $p$, the lemma is proven.
\end{proof}
%%%%%%%%%%%%%%%%%%%%%%%%%%%%%%%%%%%%%%%%%%%%%%%%%%%%%%%%%%%%%%%%%%%%%%%%%%%%5\fi
We will use Corollary \ref{cor_solution} of the following lemma in the proof of Theorem~\ref{th_group_200}.

\begin{lemma}\label{lem_existence}
A mapping $\varphi:\zn^2\to\zn^2$ given by $\varphi:(x,y)\to(a_1x+b_1y,a_2x+b_2y)$ is a group automorphism if and only if the determinant
$D=$\scalebox{.8}{$\deter{a_1}{b_1}{a_2}{b_2}$} is coprime with $n$. 
\end{lemma}

\begin{proof}
It immediately follows from the fact that the mapping $\varphi:(x,y)\to(a_1x+b_1y,a_2x+b_2y)$ is an automorphism of $\zn^2$ if and only if the matrix
\scalebox{.8}{$\mmatrix{a_1}{b_1}{a_2}{b_2}$} is invertible over the ring of integers $\Z/n\Z$.
\end{proof}

\noindent
So we have the following corollary:

\begin{corollary}\label{cor_existence}
System 
\scalebox{.8}{
$\begin{array}{ccccc}
a_{1}x & + & b_{1}y  &  = & c_1 \\
a_{2}x & + & b_{2}y  &  = & c_2 
\end{array}$ }
of linear equations over $\zn$ has a unique solution in $\zn$ if and only if the determinant $D=$\scalebox{.8}{$\deter{a_1}{b_1}{a_2}{b_2}$} is coprime with $n$.
\end{corollary}

\begin{corollary}\label{cor_solution}
Let $\alpha\in\zn^\times$ has order $k$, let $g_1(x)=(a_1x,b_1x;s_1)$, let $g_2(x)=(a_2x,b_2x;s_2)$ and let $s=s_1\cdot s_2$, $a_1,a_2,b_1,b_2,x\in\zn$, $s_1,s_2\in\Delta_k$. 
Let $g_1(x)\cdot g_2(y)=(ax+by,cx+dy;s)$, for some $a,b,c,d\in\zn$. 
If the determinant $D=$\scalebox{.8}{$\deter{a}{b}{c}{d}$} is coprime with $n$ then for each $(u,v)\in\zn^2$ there is exactly one element $(x,y)\in\zn^2$ such that 
$g_1(x)\cdot g_2(y)=(u,v;s)$.
\end{corollary}

\begin{proof}
It follows from Corollary \ref{cor_existence}.
\end{proof}
%%%%%%%%%%%%%%%%%%%%%%%%%%%%%%%%%%%%%%%%%%%%%%%%%%%%%%%%%%%%%%%%%%%%%%%%%%%%%%%%%%%%%%%%%%%%%%%%%%%%%%%%%%%%
%%%%%%%%%%%%%%%%%%%%%%%%%%%%%%%%%%%%%%%%%%%%%%%%%%%%%%%%%%%%%%%%%%%%%%%%%%%%%%%%%%%%%%%%%%%%%%%%%%%%%%%%%%%%
\section{Results}\label{sec_results}

For the rest of this paper, $P$ will denote the set of all primes of the form $p\equiv 1\pmod{10}$.

\begin{theorem}\label{th_group_200}
Let $n\in P\cup\{1\}$ and let $d = 17n-1$. Then there exists a Cayley graph of diameter two, degree $d$ and of order $\frac{200}{289}(d+1)^2$.
\end{theorem}

\begin{proof}
Let $n=p>1$ and $\alpha$ be as in Lemma \ref{lem_alpha_p}. Let $\Gamma=\Gamma(n,\alpha)=\zn^2\rtimes_\alpha\Delta_{10}$ and for $x\in\zn$ we set 
$a(x)=\ax,\ b(x)=\bx,\ c(x)=\cx,\ d(x)=\dx,\ e(x)=\ex, f(x)=\fx,\\ g(x)=\gx, h(x)=\hx$ and $k(x)=\kx$. 
We can see that 
$a^{-1}(x)=\axm$, $b^{-1}(x)=\bx$, $c^{-1}(x)=\cxm$, $d^{-1}(x)=\dxm$, $e^{-1}(x)=\exm, f^{-1}(x)=\fxm,\\ g^{-1}(x)=\gxm, h^{-1}(x)=\hxm$ and $k^{-1}(x)=\kxm$.

Let $X=\{a(x),a^{-1}(x),b(x),b^{-1}(x),c(x),c^{-1}(x),d(x),d^{-1}(x),e(x),e^{-1}(x),f(x),\\f^{-1}(x),g(x),g^{-1}(x),h(x),h^{-1}(x),k(x),k^{-1}(x)\vert x\in\zn\}$ 
and let $G=Cay(\Gamma,X)$ be the Cayley graph for the 
underlying group $\Gamma$ and the generating set $X$. Since $b(x)=b^{-1}(x)$ for each $x\in\zn$, and $a(x)=a^{-1}(x)$ exactly when $x=0$, the set $X$ has $17n-1$ elements. 
As the group 
$\Gamma$ has order $\vert\Gamma\vert=2n^2\cdot 10^2=200n^2$ and the degree of $G$ is $d=17n-1$, the corresponding Cayley graph has order $\frac{200}{289}(d+1)^2$. 
Now let a set $S\subset\Delta_{10}$ consists of elements $(i,j,0)$ such that 
$i=0,0\leq j\leq 5$; $1\leq i\leq 5,0\leq j\leq 10-i$; $6\leq i\leq 8,1\leq j\leq 9-i$ and of the elements $(i,j,1)$ such that
$i=0,j=0$ and $1\leq i\leq 9,0\leq j\leq 10-i$. We can see that $S\cup S^{-1}=\Delta_{10}$.

Since the generating set $X$ is inverse-closed, to prove that $G$ has diameter two, it is sufficient to show that every element of $\Gamma$ of the form $(x,y;s), x,y\in\zn, s\in S$
can be generated as a product of two elements from $X$. For every $s\in S$ we show that there is $g_1(x),g_2(x)\in X$ (see Corollary \ref{cor_solution}) such that the determinant 
corresponding to the product $g_1(x)\cdot g_2(y)$ is coprime with $n$. 
\def\scal{.8}

For the product of $a(x)$ and $a(y)$ we have $a(x)\cdot a(y)=\ax\cdot\ay=(x+\alpha^5 y,x+y;0,0,0)=(x-y,x+y;0,0,0)$. We have to show that for any $(u,v)\in\zn^2$ there is $x,y\in\zn$ such that $(x-y,x+y)=(u,v)$. Since the determinant \scalebox{\scal}{$D=\deter{1}{-1}{1}{1}=2$}, it is coprime with $n$, and therefore the system $x-y=u$ and $x+y=v$ has a (unique) solution for each $(u,v)\in\zn^2$. In the table below we can see how to generate the other elements of $\Gamma$.

\def\skry#1{\iffalse #1\fi} \def\novyr{}

\begingroup
\centering
\renewcommand{\arraystretch}{1.5}
\setcounter{group}{1}
\begin{longtable}{|r|r@{$\cdot$}l@{ =}r|*2{>{\renewcommand{\arraystretch}{.35}}l|}}\cline{2-5}
\multicolumn{1}{c|}{} & \multicolumn{3}{ c| }{\textbf{Product of generators}} & \multicolumn{1}{ c| }{\textbf{Determinant}}  \\\hline
\Group & $a(x)$ & $a(x)$ & $\skry{\dxm\cdot\gy=}(x-y,x+y;0,0,0)$ & 
 \scalebox{\scal}{$\deter{1}{-1}{1}{1}=2$}\\\hline
\Group & $c(x)$ & $b(y)$ & $\skry{\cxm\cdot\ey=}(x-\alpha y,x+y;1,0,0)$ & 
 \scalebox{\scal}{$\deter{1}{-\alpha}{1}{1}=\alpha+1$}\\\hline
\Group & $d^{-1}(x)$ & $g(y)$ & $\skry{\dxm\cdot\gy=}(y,x;2,0,0)$ & 
 \scalebox{\scal}{$\deter{0}{1}{1}{0}=-1$}\\\hline
\Group & $c^{-1}(x)$ & $e(y)$ & $\skry{\cxm\cdot\ey=}(-x+\alpha^3 y,\alpha^4 x-\alpha^4 y;3,0,0)$ & 
 \scalebox{\scal}{$\deter{-1}{\alpha^3}{\alpha^4}{-\alpha^4}=\alpha^4+\alpha^2$}\\\hline
\Group & $d(x)$ & $c^{-1}(y)$ & $\skry{\dx\cdot\cym=}(x-\alpha^4 y,-y;4,0,0)$ &  
 \scalebox{\scal}{$\deter{1}{-\alpha^4}{0}{-1}=-1$}\\\hline
\Group & $d(x)$ & $b(y)$ & $\skry{\dx\cdot\by=}(x+y,y;5,0,0)$ & 
 \scalebox{\scal}{$\deter{1}{1}{0}{1}=1$}\\\hline
 \Group & $b(x)$ & $c(y)$ & $\skry{\bx\cdot\cy=}(x+y,-x+y;0,1,0)$ & 
 \scalebox{\scal}{$\deter{1}{1}{-1}{1}=2$}\\\hline
\Group & $g(x)$ & $d^{-1}(y)$ & $\skry{\gx\cdot\dym=}(-y,x;0,2,0)$ & 
 \scalebox{\scal}{$\deter{0}{-1}{1}{0}=1$}\\\hline
\Group & $c(x)$ & $f^{-1}(y)$ & $\skry{\cx\cdot\fym=}(x-\alpha y,x-\alpha^3 y;0,3,0)$ & 
 \scalebox{\scal}{$\deter{1}{-\alpha}{1}{-\alpha^3}=-\alpha^3+\alpha$}\\\hline
\Group & $f(x)$ & $e^{-1}(y)$ & $\skry{\fx\cdot\eym=}(x-y,x+\alpha^2 y;0,4,0)$ & 
 \scalebox{\scal}{$\deter{1}{-1}{1}{\alpha^2}=\alpha^2+1$}\\\hline
\Group & $b(x)$ & $d(y)$ & $\skry{\bx\cdot\dy=}(x,-x+y;0,5,0)$ & 
 \scalebox{\scal}{$\deter{1}{0}{-1}{1}=1$}\\\hline
\Group & $c(x)$ & $c(y)$ & $\skry{\cx\cdot\cy=}(x+\alpha y,x+y;1,1,0)$ & 
 \scalebox{\scal}{$\deter{1}{\alpha}{1}{1}=1-\alpha$}\\\hline
\Group & $f(x)$ & $d(y)$ & $\skry{\fx\cdot\dy=}(x,x-\alpha^2 y;1,2,0)$ & 
 \scalebox{\scal}{$\deter{1}{0}{1}{-\alpha^2}=-\alpha^2$}\\\hline
\Group & $e(x)$ & $b(y)$ & $\skry{\ex\cdot\by=}(x-\alpha y,\alpha^3 x+ \alpha^3 y;1,3,0)$ & 
 \scalebox{\scal}{$\deter{1}{-\alpha}{\alpha^3}{\alpha^3}=\alpha^4+\alpha^3$}\\\hline
\Group & $e(x)$ & $c(y)$ & $\skry{\ex\cdot\cy=}(x+\alpha y,\alpha^3 x+ \alpha^3 y;1,4,0)$ & 
 \scalebox{\scal}{$\deter{1}{\alpha}{\alpha^3}{\alpha^3}=\alpha^3-\alpha^4$}\\\hline
\Group & $c(x)$ & $d(y)$ & $\skry{\cx\cdot\dy=}(x,x+y;1,5,0)$ & 
 \scalebox{\scal}{$\deter{1}{0}{1}{1}=1$}\\\hline
\Group & $g^{-1}(x)$ & $e(y)$ & $\skry{\gxm\cdot\ey=}(\alpha^3 x+\alpha y,-y;1,6,0)$ & 
 \scalebox{\scal}{$\deter{\alpha^3}{\alpha}{0}{-1}=-\alpha^3$}\\\hline
\Group & $f(x)$ & $b(y)$ & $\skry{\fx\cdot\by=}(x-\alpha y,x-\alpha^2 y;1,7,0)$ & 
 \scalebox{\scal}{$\deter{1}{-\alpha}{1}{-\alpha^2}=\alpha-\alpha^2$}\\\hline
\Group & $e(x)$ & $d(y)$ & $\skry{\ex\cdot\dy=}(x,\alpha^3 x+\alpha^3 y;1,8,0)$ & 
 \scalebox{\scal}{$\deter{1}{0}{\alpha^3}{\alpha^3}=\alpha^3$}\\\hline
\Group & $a(x)$ & $k^{-1}(y)$ & $\skry{\ax\cdot\kym=}(x-\alpha y,x+\alpha^4 y;1,9,0)$ & 
 \scalebox{\scal}{$\deter{1}{-\alpha}{1}{\alpha^4}=\alpha^4+\alpha$}\\\hline
\Group & $d(x)$ & $f(y)$ & $\skry{\dx\cdot\fy=}(x-y,y;2,1,0)$ & 
 \scalebox{\scal}{$\deter{1}{-1}{0}{1}=1$}\\\hline
\Group & $f^{-1}(x)$ & $f^{-1}(y)$ & $\skry{\fxm\cdot\fym=}(-\alpha^3 x-\alpha^2 y,\alpha^4 x-\alpha^2 y;2,2,0)$ & 
 \scalebox{\scal}{$\deter{-\alpha^3}{-\alpha^2}{\alpha^4}{-\alpha^2}=-\alpha-1$}\\\hline
\Group & $g(x)$ & $f(y)$ & $\skry{\gx\cdot\fy=}(-y,x+\alpha^2 y;2,3,0)$ & 
 \scalebox{\scal}{$\deter{0}{-1}{1}{\alpha^2}=1$}\\\hline
\Group & $c^{-1}(x)$ & $g(y)$ & $\skry{\cxm\cdot\gy=}(-x+y,\alpha^4 x;2,4,0)$ & 
 \scalebox{\scal}{$\deter{-1}{1}{\alpha^4}{0}=-\alpha^4$}\\\hline
\Group & $b(x)$ & $g(y)$ & $\skry{\bx\cdot\gy=}(x+y,-x;2,5,0)$ & 
 \scalebox{\scal}{$\deter{1}{1}{-1}{0}=1$}\\\hline
\Group & $f^{-1}(x)$ & $e^{-1}(y)$ & $\skry{\fxm\cdot\eym=}(-\alpha^3 x-\alpha^2 y,\alpha^4 x+\alpha^4 y;2,6,0)$ & 
 \scalebox{\scal}{$\deter{-\alpha^3}{-\alpha^2}{\alpha^4}{\alpha^4}=\alpha^2-\alpha$}\\\hline
\Group & $e^{-1}(x)$ & $g^{-1}(y)$ & $\skry{\exm\cdot\gym=}(-x,\alpha^4 x+\alpha^2 y;2,7,0)$ & 
 \scalebox{\scal}{$\deter{-1}{0}{\alpha^4}{\alpha^2}=-\alpha^2$}\\\hline
\Group & $a(x)$ & $h^{-1}(y)$ & $\skry{\ax\cdot\hym=}(x-\alpha^2 y,x+\alpha^4 y;2,8,0)$ & 
 \scalebox{\scal}{$\deter{1}{-\alpha^2}{1}{\alpha^4}=\alpha^4+\alpha^2$}\\\hline
\Group & $b(x)$ & $e(y)$ & $\skry{\bx\cdot\ey=}(x+\alpha^3 y,-x+y,;3,1,0)$ & 
 \scalebox{\scal}{$\deter{1}{\alpha^3}{-1}{1}=\alpha^3+1$}\\\hline
\Group & $f(x)$ & $g(y)$ & $\skry{\fx\cdot\gy=}(x+\alpha y,x;3,2,0)$ & 
 \scalebox{\scal}{$\deter{1}{\alpha}{1}{0}=-\alpha$}\\\hline
\Group & $g^{-1}(x)$ & $g^{-1}(y)$ & $\skry{\gxm\cdot\gym=}(\alpha^3 x,-\alpha^3 y;3,3,0)$ & 
 \scalebox{\scal}{$\deter{\alpha^3}{0}{0}{-\alpha^3}=\alpha$}\\\hline
\Group & $f^{-1}(x)$ & $d(y)$ & $\skry{\fxm\cdot\dy=}(-\alpha^3 x,\alpha^4 x-\alpha^4 y;3,4,0)$ & 
 \scalebox{\scal}{$\deter{-\alpha^3}{0}{\alpha^4}{-\alpha^4}=-\alpha^2$}\\\hline
\Group & $c(x)$ & $g(y)$ & $\skry{\cx\cdot\gy=}(x+\alpha y,x,;3,5,0)$ & 
 \scalebox{\scal}{$\deter{1}{\alpha}{1}{0}=-\alpha$}\\\hline
\Group & $d^{-1}(x)$ & $e(y)$ & $\skry{\dxm\cdot\ey=}(\alpha^3 y,x-y;3,6,0)$ & 
 \scalebox{\scal}{$\deter{0}{\alpha^3}{1}{-1}=-\alpha^3$}\\\hline
\Group & $h^{-1}(x)$ & $k^{-1}(y)$ & $\skry{\hxm\cdot\kym=}(\alpha^2 x-\alpha^3 y,\alpha^4 x+\alpha^2 y;3,7,0)$ & 
 \scalebox{\scal}{$\deter{\alpha^2}{-\alpha^3}{\alpha^4}{\alpha^2}=\alpha^4-\alpha^2$}\\\hline
\Group & $c(x)$ & $e(y)$ & $\skry{\cx\cdot\ey=}(x+\alpha^4 y,x+y;4,1,0)$ & 
 \scalebox{\scal}{$\deter{1}{\alpha^4}{1}{1}=1-\alpha^4$}\\\hline
\Group & $g(x)$ & $c^{-1}(y)$ & $\skry{\gx\cdot\cym=}(-\alpha^4 y,x-\alpha^2 y;4,2,0)$ & 
 \scalebox{\scal}{$\deter{0}{-\alpha^4}{1}{-\alpha^2}=\alpha^4$}\\\hline
\Group & $d(x)$ & $f^{-1}(y)$ & $\skry{\dx\cdot\fym=}(x-\alpha^4 y,-\alpha^3 y;4,3,0)$ & 
 \scalebox{\scal}{$\deter{1}{-\alpha^4}{0}{-\alpha^3}=-\alpha^3$}\\\hline
\Group & $e(x)$ & $e(y)$ & $\skry{\ex\cdot\ey=}(x+\alpha^4 y,\alpha^3 x+\alpha^3 y;4,4,0)$ & 
 \scalebox{\scal}{$\deter{1}{\alpha^4}{\alpha^3}{\alpha^3}=\alpha^3+\alpha^2$}\\\hline
\Group & $g(x)$ & $f^{-1}(y)$ & $\skry{\gx\cdot\fym=}(-\alpha^4 y,x+y;4,5,0)$ & 
 \scalebox{\scal}{$\deter{0}{-\alpha^4}{1}{1}=\alpha^4$}\\\hline
\Group & $h^{-1}(x)$ & $h^{-1}(y)$ & $\skry{\hxm\cdot\hym=}(\alpha^2 x-\alpha^4 y,\alpha^4 x+\alpha^2 y;4,6,0)$ & 
 \scalebox{\scal}{$\deter{\alpha^2}{-\alpha^4}{\alpha^4}{\alpha^2}=\alpha^4-\alpha^3$}\\\hline
\Group & $d(x)$ & $c(y)$ & $\skry{\dx\cdot\cy=}(x+y,-y;5,1,0)$ & 
 \scalebox{\scal}{$\deter{1}{1}{0}{-1}=-1$}\\\hline
\Group & $g(x)$ & $b(y)$ & $\skry{\gx\cdot\by=}(y,x+\alpha^2 y;5,2,0)$ & 
 \scalebox{\scal}{$\deter{0}{1}{1}{\alpha^2}=-1$}\\\hline
\Group & $g(x)$ & $c(y)$ & $\skry{\gx\cdot\cy=}(-y,x+\alpha^2 y;5,3,0)$ & 
 \scalebox{\scal}{$\deter{0}{-1}{1}{\alpha^2}=1$}\\\hline
\Group & $f^{-1}(x)$ & $g(y)$ & $\skry{\fxm\cdot\gy=}(-\alpha^3 x+\alpha^3 y,\alpha^4 x;5,4,0)$ & 
 \scalebox{\scal}{$\deter{-\alpha^3}{\alpha^3}{\alpha^4}{0}=\alpha^2$}\\\hline
\Group & $d(x)$ & $d(y)$ & $\skry{\dx\cdot\dy=}(x,-y;5,5,0)$ & 
 \scalebox{\scal}{$\deter{1}{0}{0}{-1}=-1$}\\\hline
\Group & $e(x)$ & $g^{-1}(y)$ & $\skry{\ex\cdot\gym=}(x,\alpha^3 x-\alpha y;6,1,0)$ & 
 \scalebox{\scal}{$\deter{1}{0}{\alpha^3}{-\alpha}=-\alpha$}\\\hline
\Group & $e^{-1}(x)$ & $f^{-1}(y)$ & $\skry{\exm\cdot\fym=}(-x+\alpha y,\alpha^4 x-\alpha^2 y;6,2,0)$ & 
 \scalebox{\scal}{$\deter{-1}{\alpha}{\alpha^4}{-\alpha^2}=\alpha^2+1$}\\\hline
\Group & $e(x)$ & $d^{-1}(y)$ & $\skry{\ex\cdot\dym=}(x+\alpha y,\alpha^3 x;6,3,0)$ & 
 \scalebox{\scal}{$\deter{1}{\alpha}{\alpha^3}{0}=-\alpha^4$}\\\hline
\Group & $b(x)$ & $f(y)$ & $\skry{\bx\cdot\fy=}(x+y,-x+y;7,1,0)$ & 
 \scalebox{\scal}{$\deter{1}{1}{-1}{1}=2$}\\\hline
\Group & $g^{-1}(x)$ & $e^{-1}(y)$ & $\skry{\gxm\cdot\eym=}(\alpha^3 x+\alpha^2 y,y;7,2,0)$ & 
 \scalebox{\scal}{$\deter{\alpha^3}{\alpha^2}{0}{1}=\alpha^3$}\\\hline
\Group & $c(x)$ & $f(y)$ & $\skry{\cx\cdot\fy=}(x+\alpha y,x+y;8,1,0)$ & 
 \scalebox{\scal}{$\deter{1}{\alpha}{1}{1}=1-\alpha$}\\\hline
\Group & $a(x)$ & $d(y)$ & $\skry{\ax\cdot\dy=}(x-y,x;0,0,1)$ & 
 \scalebox{\scal}{$\deter{1}{-1}{1}{0}=1$}\\\hline
\Group & $k(x)$ & $e^{-1}(y)$ & $\skry{\kx\cdot\eym=}(x-\alpha^4 y,\alpha x-y;1,0,1)$ & 
 \scalebox{\scal}{$\deter{1}{-\alpha^4}{\alpha}{-1}=-2$}\\\hline
\Group & $h^{-1}(x)$ & $g(y)$ & $\skry{\hxm\cdot\gy=}(\alpha^2 x,\alpha^4 x-\alpha^3 y;2,0,1)$ & 
 \scalebox{\scal}{$\deter{\alpha^2}{0}{\alpha^4}{-\alpha^3}=1$}\\\hline
\Group & $f(x)$ & $h(y)$ & $\skry{\fx\cdot\hy=}(x+\alpha^2 y,x-\alpha^2 y;3,0,1)$ & 
 \scalebox{\scal}{$\deter{1}{\alpha^2}{1}{-\alpha^2}=-2\alpha^2$}\\\hline
\Group & $k(x)$ & $c^{-1}(y)$ & $\skry{\kx\cdot\cym=}(x-\alpha^3 y,\alpha x-\alpha y;4,0,1)$ & 
 \scalebox{\scal}{$\deter{1}{-\alpha^3}{\alpha}{-\alpha}=\alpha^4-\alpha$}\\\hline
\Group & $a^{-1}(x)$ & $b(y)$ & $\skry{\axm\cdot\by=}(x-y,-x-y;5,0,1)$ & 
 \scalebox{\scal}{$\deter{1}{-1}{-1}{-1}=-2$}\\\hline
\Group & $a(x)$ & $c(y)$ & $\skry{\ax\cdot\cy=}(x-y,x+y;6,0,1)$ & 
 \scalebox{\scal}{$\deter{1}{-1}{1}{1}=2$}\\\hline
\Group & $k(x)$ & $f^{-1}(y)$ & $\skry{\kx\cdot\fym=}(x+\alpha^2 y,\alpha x-y;7,0,1)$ & 
 \scalebox{\scal}{$\deter{1}{\alpha^2}{\alpha}{-1}=-\alpha^3-1$}\\\hline
\Group & $g^{-1}(x)$ & $a(y)$ & $\skry{\gxm\cdot\ay=}(\alpha^3 x-\alpha^3 y,-y;8,0,1)$ & 
 \scalebox{\scal}{$\deter{\alpha^3}{-\alpha^3}{0}{-1}=-\alpha^3$}\\\hline
\Group & $e(x)$ & $h^{-1}(y)$ & $\skry{\ex\cdot\hym=}(x-y,\alpha^3 x-y;9,0,1)$ & 
 \scalebox{\scal}{$\deter{1}{-1}{\alpha^3}{-1}=\alpha^3-1$}\\\hline
\Group & $k^{-1}(x)$ & $g(y)$ & $\skry{\kxm\cdot\gy=}(\alpha x,\alpha^4 x-\alpha^4 y;1,1,1)$ & 
 \scalebox{\scal}{$\deter{\alpha}{0}{\alpha^4}{-\alpha^4}=1$}\\\hline
\Group & $f(x)$ & $a(y)$ & $\skry{\fx\cdot\ay=}(x+\alpha y,x-\alpha^2 y;1,2,1)$ & 
 \scalebox{\scal}{$\deter{1}{\alpha}{1}{-\alpha^2}=-\alpha^2-\alpha$}\\\hline
\Group & $c^{-1}(x)$ & $k(y)$ & $\skry{\cxm\cdot\ky=}(-x+\alpha y,\alpha^4 x-\alpha^4 y;1,3,1)$ & 
 \scalebox{\scal}{$\deter{-1}{\alpha}{\alpha^4}{-\alpha^4}=\alpha^4+1$}\\\hline
\Group & $b(x)$ & $k(y)$ & $\skry{\bx\cdot\ky=}(x+\alpha y,-x+y;1,4,1)$ & 
 \scalebox{\scal}{$\deter{1}{\alpha}{-1}{1}=\alpha+1$}\\\hline
\Group & $c(x)$ & $a(y)$ & $\skry{\cx\cdot\ay=}(x+\alpha y,x+y;1,5,1)$ & 
 \scalebox{\scal}{$\deter{1}{\alpha}{1}{1}=1-\alpha$}\\\hline
\Group & $f^{-1}(x)$ & $h^{-1}(y)$ & $\skry{\fxm\cdot\hym=}(-\alpha^3 x-\alpha^2 y,\alpha^4 x+\alpha y;1,6,1)$ & 
 \scalebox{\scal}{$\deter{-\alpha^3}{-\alpha^2}{\alpha^4}{\alpha}=-\alpha^4+\alpha$}\\\hline
\Group & $h(x)$ & $g^{-1}(y)$ & $\skry{\hx\cdot\gym=}(x-\alpha y,\alpha x;1,7,1)$ & 
 \scalebox{\scal}{$\deter{1}{-\alpha}{\alpha}{0}=\alpha^2$}\\\hline
\Group & $e(x)$ & $a(y)$ & $\skry{\ex\cdot\ay=}(x+\alpha^3 y,\alpha^3 x+\alpha^3 y;1,8,1)$ & 
 \scalebox{\scal}{$\deter{1}{\alpha^3}{\alpha^3}{\alpha^3}=\alpha^3+\alpha$}\\\hline
\Group & $k^{-1}(x)$ & $d(y)$ & $\skry{\kxm\cdot\dy=}(\alpha x-\alpha y,\alpha^4 x;1,9,1)$ & 
 \scalebox{\scal}{$\deter{\alpha}{-\alpha}{\alpha^4}{0}=1$}\\\hline
\Group & $f(x)$ & $k(y)$ & $\skry{\fx\cdot\ky=}(x+\alpha^2 y,x-\alpha^2 y;2,1,1)$ & 
 \scalebox{\scal}{$\deter{1}{\alpha^2}{1}{-\alpha^2}=-2\alpha^2$}\\\hline
\Group & $c^{-1}(x)$ & $h(y)$ & $\skry{\cxm\cdot\hy=}(-x+\alpha y,\alpha^4 x-\alpha^4 y;2,2,1)$ & 
 \scalebox{\scal}{$\deter{-1}{\alpha}{\alpha^4}{-\alpha^4}=\alpha^4+1$}\\\hline
\Group & $b(x)$ & $h(y)$ & $\skry{\bx\cdot\hy=}(x+\alpha y,-x+y;2,3,1)$ & 
 \scalebox{\scal}{$\deter{1}{\alpha}{-1}{1}=\alpha+1$}\\\hline
\Group & $c(x)$ & $k(y)$ & $\skry{\cx\cdot\ky=}(x+\alpha^2 y,x+y;2,4,1)$ & 
 \scalebox{\scal}{$\deter{1}{\alpha^2}{1}{1}=1-\alpha^2$}\\\hline
\Group & $f^{-1}(x)$ & $k^{-1}(y)$ & $\skry{\fxm\cdot\kym=}(-\alpha^3 x-\alpha^2 y,\alpha^4 x+y;2,5,1)$ & 
 \scalebox{\scal}{$\deter{-\alpha^3}{-\alpha^2}{\alpha^4}{1}=-\alpha^3-\alpha$}\\\hline
\Group & $k(x)$ & $g^{-1}(y)$ & $\skry{\kx\cdot\gym=}(x-\alpha^2 y,\alpha x;2,6,1)$ & 
 \scalebox{\scal}{$\deter{1}{-\alpha^2}{\alpha}{0}=\alpha^3$}\\\hline
\Group & $e(x)$ & $k(y)$ & $\skry{\ex\cdot\ky=}(x+\alpha^2 y,\alpha^3 x+\alpha^3 y;2,7,1)$ & 
 \scalebox{\scal}{$\deter{1}{\alpha^2}{\alpha^3}{\alpha^3}=\alpha^3+1$}\\\hline
\Group & $h^{-1}(x)$ & $d(y)$ & $\skry{\hxm\cdot\dy=}(\alpha^2 x-\alpha^2 y,\alpha^4 x;2,8,1)$ & 
 \scalebox{\scal}{$\deter{\alpha^2}{-\alpha^2}{\alpha^4}{0}=-\alpha$}\\\hline
\Group & $h(x)$ & $c^{-1}(y)$ & $\skry{\hx\cdot\cym=}(x-\alpha^3 y,\alpha x-y;3,1,1)$ & 
 \scalebox{\scal}{$\deter{1}{-\alpha^3}{\alpha}{-1}=\alpha^4-1$}\\\hline
\Group & $h(x)$ & $b(y)$ & $\skry{\hx\cdot\by=}(x+\alpha^3 y,\alpha x-\alpha^2 y;3,2,1)$ & 
 \scalebox{\scal}{$\deter{1}{\alpha^3}{\alpha}{-\alpha^2}=-\alpha^4-\alpha^2$}\\\hline
\Group & $c(x)$ & $h(y)$ & $\skry{\cx\cdot\hy=}(x+\alpha^2 y,x+y;3,3,1)$ & 
 \scalebox{\scal}{$\deter{1}{\alpha^2}{1}{1}=1-\alpha^2$}\\\hline
\Group & $f^{-1}(x)$ & $a^{-1}(y)$ & $\skry{\fxm\cdot\aym=}(-\alpha^3 x-\alpha^3 y,\alpha^4 x-\alpha^4 y;3,4,1)$ & 
 \scalebox{\scal}{$\deter{-\alpha^3}{-\alpha^3}{\alpha^4}{-\alpha^4}=-2\alpha^2$}\\\hline
\Group & $a(x)$ & $g^{-1}(y)$ & $\skry{\ax\cdot\gym=}(x-\alpha^3 y,x;3,5,1)$ & 
 \scalebox{\scal}{$\deter{1}{-\alpha^3}{1}{0}=\alpha^3$}\\\hline
\Group & $e(x)$ & $h(y)$ & $\skry{\ex\cdot\hy=}(x+\alpha^2 y,\alpha^3 x+\alpha^3 y;3,6,1)$ & 
 \scalebox{\scal}{$\deter{1}{\alpha^2}{\alpha^3}{\alpha^3}=\alpha^3+1$}\\\hline
\Group & $h(x)$ & $d^{-1}(y)$ & $\skry{\hx\cdot\dym=}(x,\alpha x+\alpha^2 y;3,7,1)$ & 
 \scalebox{\scal}{$\deter{1}{0}{\alpha}{\alpha^2}=\alpha^2$}\\\hline
\Group & $k(x)$ & $b(y)$ & $\skry{\kx\cdot\by=}(x+\alpha^4 y,\alpha x-\alpha y;4,1,1)$ & 
 \scalebox{\scal}{$\deter{1}{\alpha^4}{\alpha}{-\alpha}=1-\alpha$}\\\hline
\Group & $h(x)$ & $c(y)$ & $\skry{\hx\cdot\cy=}(x+\alpha^3 y,\alpha x+\alpha^2 y;4,2,1)$ & 
 \scalebox{\scal}{$\deter{1}{\alpha^3}{\alpha}{\alpha^2}=\alpha^2-\alpha^4$}\\\hline
\Group & $f^{-1}(x)$ & $k(y)$ & $\skry{\fxm\cdot\ky=}(-\alpha^3 x+\alpha^4 y,\alpha^4 x-\alpha^4 y;4,3,1)$ & 
 \scalebox{\scal}{$\deter{-\alpha^3}{\alpha^4}{\alpha^4}{-\alpha^4}=\alpha^3-\alpha^2$}\\\hline
\Group & $k^{-1}(x)$ & $g^{-1}(y)$ & $\skry{\kxm\cdot\gym=}(\alpha x-\alpha^4 y,\alpha^4 x;4,4,1)$ & 
 \scalebox{\scal}{$\deter{\alpha}{-\alpha^4}{\alpha^4}{0}=-\alpha^3$}\\\hline
\Group & $h(x)$ & $e(y)$ & $\skry{\hx\cdot\ey=}(x+\alpha^3 y,\alpha x-y;4,5,1)$ & 
 \scalebox{\scal}{$\deter{1}{\alpha^3}{\alpha}{-1}=-\alpha^4-1$}\\\hline
\Group & $k(x)$ & $d^{-1}(y)$ & $\skry{\kx\cdot\dym=}(x,\alpha x+\alpha y;4,6,1)$ & 
 \scalebox{\scal}{$\deter{1}{0}{\alpha}{\alpha}=\alpha$}\\\hline
\Group & $k(x)$ & $c(y)$ & $\skry{\kx\cdot\cy=}(x+\alpha^4 y,\alpha x+\alpha y;5,1,1)$ & 
 \scalebox{\scal}{$\deter{1}{\alpha^4}{\alpha}{\alpha}=\alpha+1$}\\\hline
\Group & $f^{-1}(x)$ & $h(y)$ & $\skry{\fxm\cdot\hy=}(-\alpha^3 x+\alpha^4 y,\alpha^4 x-\alpha^4 y;5,2,1)$ & 
 \scalebox{\scal}{$\deter{-\alpha^3}{\alpha^4}{\alpha^4}{-\alpha^4}=\alpha^3-\alpha^2$}\\\hline
\Group & $h^{-1}(x)$ & $g^{-1}(y)$ & $\skry{\hxm\cdot\gym=}(\alpha^2 x+y,\alpha^4 x;5,3,1)$ & 
 \scalebox{\scal}{$\deter{\alpha^2}{1}{\alpha^4}{0}=-\alpha^4$}\\\hline
\Group & $k(x)$ & $e(y)$ & $\skry{\kx\cdot\ey=}(x+\alpha^4 y,\alpha x+\alpha^4 y;5,4,1)$ & 
 \scalebox{\scal}{$\deter{1}{\alpha^4}{\alpha}{\alpha^4}=\alpha^4+1$}\\\hline
\Group & $d(x)$ & $a(y)$ & $\skry{\dx\cdot\ay=}(x-y,y;5,5,1)$ & 
 \scalebox{\scal}{$\deter{1}{-1}{0}{1}=1$}\\\hline
\Group & $h(x)$ & $f^{-1}(y)$ & $\skry{\hx\cdot\fym=}(x+\alpha y,\alpha x-\alpha y;6,1,1)$ & 
 \scalebox{\scal}{$\deter{1}{\alpha}{\alpha}{-\alpha}=-\alpha^2-\alpha$}\\\hline
\Group & $g^{-1}(x)$ & $h^{-1}(y)$ & $\skry{\gxm\cdot\hym=}(\alpha^3 x+\alpha^2 y,-\alpha^2 y;6,2,1)$ & 
 \scalebox{\scal}{$\deter{\alpha^3}{\alpha^2}{0}{-\alpha^2}=1$}\\\hline
\Group & $a(x)$ & $e(y)$ & $\skry{\ax\cdot\ey=}(x-y,x+\alpha^3 y;6,3,1)$ & 
 \scalebox{\scal}{$\deter{1}{-1}{1}{\alpha^3}=\alpha^3+1$}\\\hline
\Group & $d(x)$ & $k(y)$ & $\skry{\dx\cdot\ky=}(x-\alpha y,y;6,4,1)$ & 
 \scalebox{\scal}{$\deter{1}{-\alpha}{0}{1}=1$}\\\hline
\Group & $g^{-1}(x)$ & $k^{-1}(y)$ & $\skry{\gxm\cdot\kym=}(\alpha^3 x+\alpha^2 y,-\alpha y;7,1,1)$ & 
 \scalebox{\scal}{$\deter{\alpha^3}{\alpha^2}{0}{-\alpha}=-\alpha^4$}\\\hline
\Group & $k^{-1}(x)$ & $e(y)$ & $\skry{\kxm\cdot\ey=}(\alpha x-\alpha y,\alpha^4 x+\alpha^2 y;7,2,1)$ & 
 \scalebox{\scal}{$\deter{\alpha}{-\alpha}{\alpha^4}{\alpha^2}=\alpha^3+1$}\\\hline
\Group & $d(x)$ & $h(y)$ & $\skry{\dx\cdot\hy=}(x-\alpha y,y;7,3,1)$ & 
 \scalebox{\scal}{$\deter{1}{-\alpha}{0}{1}=1$}\\\hline
\Group & $h^{-1}(x)$ & $e(y)$ & $\skry{\hxm\cdot\ey=}(\alpha^2 x-\alpha^2 y,\alpha^4 x+\alpha y;8,1,1)$ & 
 \scalebox{\scal}{$\deter{\alpha^2}{-\alpha^2}{\alpha^4}{\alpha}=\alpha^3+\alpha$}\\\hline
\Group & $h(x)$ & $d(y)$ & $\skry{\hx\cdot\dy=}(x+\alpha^3 y,\alpha x;8,2,1)$ & 
 \scalebox{\scal}{$\deter{1}{\alpha^3}{\alpha}{0}=-\alpha^4$}\\\hline
\Group & $k(x)$ & $d(y)$ & $\skry{\kx\cdot\dy=}(x+\alpha^4 y,\alpha x;9,1,1)$ & 
 \scalebox{\scal}{$\deter{1}{\alpha^4}{\alpha}{0}=1$}\\\hline
%\caption{Caption of the \texttt{longtable} environment.}
\caption{Products and determinants.}
\label{longtable}
\end{longtable}
\endgroup

We can see that all the determinants are from the set $\Lambda$ (Lemma \ref{lem_Lambda}), that is, coprime with $n$.
For $n=1$ the group $\Gamma$ consists of elements $(0,0;s)$, $s\in\Delta_{10}$ and the corresponding Cayley graph $G=Cay(\Delta_{10},X)$ has diameter two, degree 16 and order 200, which gives $\vert G\vert=0.78125d^2$.
\end{proof}

In the following lemma we use explicit estimates for the distribution of primes in arithmetic progressions to show that in a "short" interval there is always a prime $p$ 
of the form $p=10s+1$.
We will need the lemma in the proof of Main Theorem \ref{th_main}.

\begin{lemma}\label{lem_interval}
Let $x$ be any real number \iffalse{$x\geq\xa$}\fi $x\geq\xb$. Then there is a prime $p$ of the form $p=10s+1$ such that $p\in (x; \deltA x\rangle$. 
\end{lemma}

\begin{proof}
We will prove the lemma separately for i) $x\geq\xd$, ii) $\xc\leq x\leq \xd$  and iii) $\xb\leq x\leq \xc$. \iffalse and for iv) $\xa\leq x\leq \xb$.\fi 
In the proof of i) and ii) we follow the paper \cite{CH2012} which is based on the work \cite{RR1996}. As usual, let $\theta(x)$ denote the first Chebyshev function defined by 
$\theta(x)=\sum\limits_{p\leq x}\log{p}$, where the sum is over all primes not exceeding $x$. If $k$ and $l$ are relatively prime and  $l\leq k$ the function $\theta(x;k,l)$ is defined as 
\begin{equation}\label{theta}
\theta(x;k,l)=\sum\limits_{\substack{p\leq x\\ p\equiv l \mmod k}}\log{p}.
\end{equation}

\noindent i) Let $x\geq\xd$. Theorem \cite[Theorem 1]{RR1996} says that for each positive real number $x$ and for each coprime $k,l,\ l\leq k$, there is an $\epsilon=\epsilon(x,k)$ such that
\begin{equation}\label{theta_th}
\max\limits_{1\leq y\leq x}\left|\theta(y;k,l)-\frac{y}{\varphi(k)}\right|\leq\epsilon \frac{x}{\varphi(k)}.
\end{equation}
The values of $\epsilon$ for various quantities of $x$ and $k$ are given by \cite[Table 1]{RR1996}. 
From the previous inequality it follows that 
\begin{equation}\label{theta_0}
-\frac{\epsilon}{\varphi(k)}x\leq\theta(x;k,l)-\frac{x}{\varphi(k)}\leq \frac{\epsilon}{\varphi(k)}x.
\end{equation}
Now let $\delta\geq0$, let $x'=(1+\delta)x$ let $\theta=\theta(x;k,l)$ and let $\theta'=\theta(x';k,l)$. 
From \eqref{theta_0} we have 
$-\left(\theta-\frac{x}{\varphi(k)}\right)\geq-\frac{\epsilon}{\varphi(k)}x$ and $\theta'-\frac{x'}{\varphi(k)}\geq-\frac{\epsilon}{\varphi(k)}x'$. It follows that 
$\theta'-\theta\geq \frac{x}{\varphi(k)}[(1-\epsilon)(1+\delta)-(1+\epsilon)]$. We see that if 
\begin{equation}\label{delta}
(1-\epsilon)(1+\delta)-(1+\epsilon)>0
\end{equation} 
then $\theta'-\theta>0$ and consequently the interval $(x,x'\rangle$ contains a prime $p\equiv l\mmod k$. Solving the inequality \eqref{delta} one can obtain 
$\delta>\frac{2\epsilon}{1-\epsilon}$. For $k=10$ and $x\geq\xd$ (\cite[Table 1]{RR1996}) we have $\epsilon=\eps$. 
Since $\frac{2\cdot\eps}{1-\eps}=0.00558556$, choosing $\delta=\delt$ we have: for $x\geq\xd$ there is a prime $p\equiv 1\pmod{10}$ in the interval $(x,\deltA x\rangle$.
\\\\
\noindent ii) Let $\xc\leq x\leq\xd$. Theorem \cite[Theorem 2]{RR1996} gives for each positive real number $x\leq\xd$ and for coprimes $k,l$ and constant $\gamma=\gamma(k)$ (given for various values of $k$ by table \cite[Table 2]{RR1996}) the inequality
\begin{equation}\label{theta_gam}
\max\limits_{1\leq y\leq x}\left|\theta(y;k,l)-\frac{y}{\varphi(k)}\right|\leq\gamma\sqrt{x}.
\end{equation}
From \eqref{theta_gam} we get inequalities $-\left(\theta-\frac{x}{\varphi(k)}\right)\geq-\gamma\sqrt{x}$ and $\theta'-\frac{x'}{\varphi(k)}\geq-\gamma\sqrt{x'}$. It follows that 
$\theta'-\theta\geq \sqrt{x}[\frac{\delta}{\varphi(k)}\sqrt{x}-\gamma(1+\sqrt{1+\delta})]$. Solving the inequality 
\begin{equation}\label{gamma}
\frac{\delta}{\varphi(k)}\sqrt{x}-\gamma(1+\sqrt{1+\delta})>0
\end{equation}
we see that if $x>\left[\frac{\varphi(k)\gamma(\sqrt{1+\delta}+1)}{\delta}\right]^2$ then the interval $(x,x'\rangle=(x,(1+\delta)x\rangle$ contains a prime $p\equiv l\pmod k$ for each $x\leq\xd$. For $k=10$ we have $\varphi(k)=4$ and (from the table \cite[Table 2]{RR1996}) $\gamma(k)=\gam$. 
Since $\left[\frac{4\cdot\gam\cdot(\sqrt{1+\delt}+1)}{\delt}\right]^2 =4104075.014974$, we have: for $\xc\leq x\leq\xd$ there is a prime $p\equiv 1\pmod{10}$ in the interval 
$(x,\deltA x\rangle$.
\\\\
\noindent iii) Let $\xb\leq x\leq\xc$. It is easy to check (e.g. by GAP) that the difference between every two consecutive primes $\xb\leq p_1=10s_1+1<p_2=10s_2+1\leq\xc$ is always 
less than $1.005p_1$. That is, for $\xb\leq x\leq\xc$ there is a prime $p\equiv 1\pmod{10}$ in the interval $(x,\deltA x\rangle$.
\iffalse
\\\\
\noindent iv) Let $\xa\leq x\leq\xb$. It was checked (by GAP) that the difference between every two consecutive elements $\xa\leq n_1<n_2\leq\xb$ of the set $P$ is always 
less than $1.005n_1$. That is, for $\xa\leq x\leq\xb$ there is an element of $P$ in the interval $(x,\deltA x\rangle$.
\fi
\end{proof}

\begin{theorem}[Main Theorem]\label{th_main}
$C(d,2)>\koef d^2$ for every integer \nolinebreak \iffalse $d\geq \dsharp$\fi $d\geq\dsharpPrime$.
\end{theorem}

\begin{proof}
Let $d=17n-1+r$, $r\in\{0,1,\dots,16\}$ (that is $n\geq\xasharp$), let $p$ be the greatest prime number of the form $p=10s+1$ such that $p\leq n$ and let $d'=17p-1$. 
By Theorem \ref{th_group_200} 
there is a Cayley graph $G'$ of diameter two, degree $d'$ and of order $\frac{200}{289}(d'+1)^2$. Adding $d-d'$ additional generators to the generating set of the corresponding Cayley graph we obtain a Cayley graph $G$ of degree $d$ and of the same order as $G'$. By Lemma \ref{lem_interval}, for the number $n$ we have $p\leq n<(1+\delta)p$ where $\delta=\delt$. From the inequality $n\leq (1+\delta)p$ we get $17n-1+r\leq (17p-1)+11\delta p+r$ that is 
$d\leq d'(1+\delta) + (\delta+r)$ and consequently $d'+1\geq\frac{d+1-r}{1+\delta}$.
It follows that
$C(d,2)=C(d',2)=\frac{200}{289}(d'+1)^2\geq\frac{200}{289}\left(\frac{d+1-r}{1+\delta}\right)^2\geq \frac{200}{289(1+\delta)^2}(1-\frac{15}{d})^2d^2$.
Since $d\geq\dsharpPrime$, we have 
$C(d,2)\geq\frac{200}{289(1+\delta)^2}\left(1-\frac{15}{\dsharpPrime}\right)^2d^2\geq 0.684317d^2>\koef d^2$ for every integer 
$d\geq\dsharpPrime$. 
\end{proof}

%%%%%%%%%%%%%%%%%%%%%%%%%%%%%%%%%%%%%%%%%%%%%%%%%%%%%%%%%%%%%%%%%%%%%%%%%%%%%%%%%%%%%%%%%%%%%%%%%%%%%%%%%%%%
%%%%%%%%%%%%%%%%%%%%%%%%%%%%%%%%%%%%%%%%%%%%%%%%%%%%%%%%%%%%%%%%%%%%%%%%%%%%%%%%%%%%%%%%%%%%%%%%%%%%%%%%%%%%
\section{Conclusion and remarks}\label{sec_remarks}

We have given a construction of Cayley graphs of diameter two with $C(d,2)\geq c d^2$ for every degree $d$, where $c>0$ is a positive constant. Let $D$ be a degree, for which a Cayley graph constructed in \cite{SS2012} is defined. From the analysis in \cite{A2014} it follows, that in the interval $(D;2D)$ the construction with "$c$" gives better results for degrees $d\in(\frac{1}{\sqrt{c}}D;2D)$. Let $c'=\frac{1}{\sqrt{c}}$. It follows that (roughly speaking) the ratio of number of better results for "$c$-construction" to the number of all constructions (up to a degree $2^{s+1}D$) is given by $\frac{(2D-c'D)+(4D-2c'D)+\dots+(2^{s+1}D-2^sc'D)}{2^{s+1}D}$ which has the limit $2-\frac{1}{\sqrt{c}}$ as $D$ and $s$ tend to infinity.
On the other hand, if one counts the number of better results for "$c$-construction" in intervals $(c'D;2c'D)$, the ratio is\\ $\frac{(2D-c'D)+(4D-2c'D)+\dots+(2^{s+1}D-2^sc'D)}{2^{s+1}c'D}$ which has the limit $2\sqrt{c}-1$. It follows that the percentage of better results for "$c$-construction" is between $2\sqrt{c}-1$ and $2-\frac{1}{\sqrt{c}}$. 
For $c=\frac{200}{289}$ we have the interval $(\frac{20\sqrt{2}}{11}-1;2-\frac{17\sqrt{2}}{20})$ which is approximately $(0.66378;0.797918)$.
Below we can see the table of intervals of degrees (the computation was performed by GAP) for which our construction gives better results as the construction in \cite{SS2012}. In the second and third column of the table there is the ratio of the number of degrees when our construction gives better results as those in \cite{SS2012} to the number of all degrees for intervals $(c'D;2c'D)$ and $(D;2D)$, respectively, for degrees $D$ from $10^3$ do $10^{14}$. We can see that these values approach the values 
$\frac{20\sqrt{2}}{11}-1\approx 0.66378$ and $2-\frac{17\sqrt{2}}{20}\approx 0.797918$, respectively.
 
\begin{center}
\begin{longtable}{|c||c|c|}
%\caption{The numerical results for numbers of degrees when our construction is better as that in \cite{SS2012}}
%\\
\hline
\textbf{Degrees for which our construction is better} & \textbf{Min ratio} & \textbf{Max ratio} \cr\hline\hline
$\langle 2566;4345\rangle$ &0.702834 & 0.8\cr\hline $\langle 4946;8569\rangle$ &0.718946 & 0.828585\cr\hline $\langle 9876;
16889\rangle$ &0.715155 & 0.835702\cr\hline $\langle 19736;33529\rangle$ &0.707326 & 0.832359\cr\hline $\langle 39456;
66553\rangle$ &0.697199 & 0.826501\cr\hline $\langle 78896;132601\rangle$ &0.689772 & 0.819844\cr\hline $\langle 157606;
264185\rangle$ &0.683042 & 0.814929\cr\hline $\langle 315196;527353\rangle$ &0.678087 & 0.810558\cr\hline $\langle 630376;
1052665\rangle$ &0.674094 & 0.807227\cr\hline $\langle 1260566;2103289\rangle$ &0.671315 & 0.804675\cr\hline $\langle 2521116;
4202489\rangle$ &0.669141 & 0.802819\cr\hline $\langle 5042046;8400889\rangle$ &0.667666 & 0.801425\cr\hline $\langle 10083906;
16793593\rangle$ &0.666532 & 0.800446\cr\hline $\langle 20167626;33579001\rangle$ &0.665764 & 0.799718\cr\hline $\langle 40335236;
67141625\rangle$ &0.665177 & 0.799208\cr\hline $\langle 80670456;134266873\rangle$ &0.664783 & 0.798831\cr\hline $\langle 161340726;
268500985\rangle$ &0.664485 & 0.79857\cr\hline $\langle 322681436;536969209\rangle$ &0.664285 & 0.798378\cr\hline $\langle 645362686;
1073872889\rangle$ &0.664134 & 0.798246\cr\hline $\langle 1290725356;2147680249\rangle$ &0.664034 & 0.798149\cr\hline $\langle 2581450526;
4295229433\rangle$ &0.663958 & 0.798083\cr\hline $\langle 5162900866;8590327801\rangle$ &0.663907 & 0.798034\cr\hline $\langle 10325801716;
17180393465\rangle$ &0.663869 & 0.798001\cr\hline $\langle 20651603416;34360524793\rangle$ &0.663844 & 0.797976\cr\hline $\langle 41303206816;
68720525305\rangle$ &0.663825 & 0.79796\cr\hline $\langle 82606413616;137440526329\rangle$ &0.663812 & 0.797947\cr\hline $\langle 165212827216;
274880004089\rangle$ &0.663803 & 0.797939\cr\hline $\langle 330425654416;549758959609\rangle$ &0.663797 & 0.797933\cr\hline $\langle 660851308816;
1099515822073\rangle$ &0.663792 & 0.797929\cr\hline $\langle 1321702617616;2199029547001\rangle$ &0.663789 & 0.797926\cr\hline $\langle 2643405235216;
4398054899705\rangle$ &0.663786 & 0.797924\cr\hline $\langle 5286810470416;8796105605113\rangle$ &0.663785 & 0.797922\cr\hline $\langle 10573620940816;
17592202821625\rangle$ &0.663783 & 0.797921\cr\hline $\langle 21147241881447;35184397254649\rangle$ &0.663783 & 0.79792\cr\hline $\langle 42294483762876;
70368777732089\rangle$ &0.663782 & 0.79792\cr\hline % $\langle 84588967525736;140737538686969\rangle$ &0.797919 & 0.797919\cr\hline
\caption{The numerical results for numbers of degrees when our construction is better as that in \cite{SS2012}}
\label{tab_density}
\end{longtable}
\end{center}
%%%%%%%%%%%%%%%%%%%%%%%%%%%%%%%%%%%%%%%%%%%%%%%%%%%%%%%%%%%%%%%%%%%%%%%%%%%%%%%%%%%%%%%%%%%%%%%%%%%%%%%%%%%%
%%%%%%%%%%%%%%%%%%%%%%%%%%%%%%%%%%%%%%%%%%%%%%%%%%%%%%%%%%%%%%%%%%%%%%%%%%%%%%%%%%%%%%%%%%%%%%%%%%%%%%%%%%%%
\section*{Acknowledgements}

I would like to express my gratitude to the referees for all the valuable and constructive comments.\\
%The research was supported by VEGA Research Grant No. 1/0811/14 and by the ERDF (ITMS: 26220220179). 
%The research was supported by VEGA Research Grant No. 1/0811/14 and by the project ITMS 26220220179 of Operational Programme 'Research \& Development' funded by the ERDF.
The research was supported by VEGA Research Grant No. 1/0811/14 and by 
the Operational Programme 'Research \& Development' funded by the European Regional Development Fund through implementation of the project ITMS 26220220179.
%%%%%%%%%%%%%%%%%%%%%%%%%%%%%%%%%%%%%%%%%%%%%%%%%%%%%%%%%%%%%%%%%%%%%%%%%%%%%%%%%%%%%%%%%%%%%%%%%%%%%%%%%%%%
%%%%%%%%%%%%%%%%%%%%%%%%%%%%%%%%%%%%%%%%%%%%%%%%%%%%%%%%%%%%%%%%%%%%%%%%%%%%%%%%%%%%%%%%%%%%%%%%%%%%%%%%%%%%


\begin{thebibliography}{99}
%%%%%%%%%%%%%%%%%%%%%%%%%%%%%%%%%%%%%%%%%%%%%%%%%%%%%%%%%%%%%%%%%%%%%%%%%%%%%%%%%%%%
\bibitem{A2014} M. Abas, \textit{Cayley graphs of diameter two and any degree with order half of the Moore bound}, 
Discrete Applied Mathematics, \textbf{173}, (2014), 1--7

\bibitem{A2015} M. Abas, \textit{On Record Cayley Graphs of Diameter Two}, 
Submitted for publication. Available as \href{http://arxiv.org/abs/1509.00842}{arXiv:1509.00842}

\bibitem{B1996} W. G. Brown, \textit{On graphs that do not contain a Thompsen graph}, Canad. Math. Bull., \textbf{9}, (1996), 281--285

\bibitem{C1999} P. J. Cameron, \textit{Permutation groups}, Cambridge University Press, (1999)

\bibitem{CH2012} J. Cullinan, F. Hajir, \textit{Primes of prescribed congruence class in short intervals}, Integers, 12 (2012), Paper A56, 4 p., electronic only

\bibitem{D1973} R. M. Damerell, \textit{On Moore graphs}, Proc. Cambridge Phil. Soc. 74, (1973), 227–-236.

\bibitem{EFH1980} P. Erd\"os, S. Fajtlowicz, A. J. Hoffman, \textit{Maximum degree in graphs of diameter 2}, 
Networks, \textbf{10}, (1980), 87--90

\iffalse
\bibitem{GAP} GAP - Groups, Algorithms, Programming - a System for Computational Discrete Algebra, \textit{www.gap-system.org}
\fi

\bibitem{HS1960} A. J. Hoffman, R. R. Singleton, \textit{On Moore graphs with diameter 2 and 3}, IBM J. Res. Develop. 4, (1960), 497–-504.

\bibitem{MS2010} M. Ma\v caj, J. \v Sir\'a\v n, \textit{Search for properties of the missing Moore graph}, 
Linear Algebra and its Applications, \textbf{432}, No. 9, (2010), 2381--398

\bibitem{MMS1998} B. D. McKay, M. Miller, J. \v Sir\'a\v n, \textit{A note on large graphs of diameter two and given maximum degree}, 
J. Combin. Theory Ser. B, \textbf{74}, (1998), 110--118

\bibitem{MS2005} M. Miller,  J. \v Sir\'a\v n, \textit{Moore graphs and beyond: a survey of the degree/diameter problem}, Electron. J. Combin. (2013) DS14

\bibitem{RR1996} O. Ramar\'e, R. Rumely, \textit{Primes in arithmetic progressions}, Mathematics of Computation, 65 (213), (1996), 397--425. 

\iffalse
\bibitem{S2001} J. \v Siagiov\' a, \textit{A note on the McKay-Miller-\v Sir\'a\v n graphs}, 
Journal of Combinatorial Theory, Series B, \textbf{81}, (2001), 205--208
\fi

\bibitem{SS2005} J. \v Siagiov\' a, J. \v Sir\'a\v n, \textit{A note on large Cayley graphs of diameter two and given degree}, 
Discrete Mathematics, \textbf{305}, No. 1-3, (2005), 379--382

\bibitem{SS2012} J. \v Siagiov\' a, J. \v Sir\'a\v n, \textit{Approaching the Moore bound for diameter two by Cayley graphs}, 
Journal of Combinatorial Theory, Series B, \textbf{102}, No. 2, (2012), 470--473

\bibitem{SSZ2010} J. \v Sir\'a\v n, J. \v Siagiov\' a, M. \v Zd\'\i malov\'a, \textit{Large graphs of diameter two and given degree}, In: Proc. IWONT 2010, Univ. Politecnica de Catalunya., 347--359.

\iffalse
\bibitem{wiki_Marston_10} 
\href{http://combinatoricswiki.org/wiki/Description_of_optimal_Cayley_graphs_found_by_Marston_Conder#List_of_generators_for_graphs_of_degree_10}%
{Description of provably largest Cayley graphs for the degree-diameter problem at\\%
combinatoricswiki.org/wiki/Description\_of\_optimal\_Cayley\_graphs\_found\_by\_Marston\_Conder}
\fi

\iffalse
\bibitem{wiki_Table} 
\href{http://combinatoricswiki.org/wiki/The_Degree_Diameter_Problem_for_Cayley_Graphs}%
{Largest known Cayley graphs of for degree-diameter problem at\\%
combinatoricswiki.org/wiki/The\_Degree\_Diameter\_Problem\_for\_Cayley\_Graphs}
\fi

\iffalse
http://combinatoricswiki.org/wiki/Description_of_optimal_Cayley_graphs_found_by_Marston_Conder#List_of_generators_for_graphs_of_degree_10
%#########################################################################################################################################################

\bibitem{S1861} H. J. S. Smith, org, On Systems of Linear Indeterminate Equations and Congruences,
Philosophical Transactions of the Royal Society of London, Vol. 151, No. (1861), 293--326,
Published by: The Royal Society, 
Stable URL: \href{http://www.jstor.org/stable/108738}{www.jstor.org/stable/108738}

\bibitem{A1992} M. Aider, R\'eseaux d'interconnexion bipartis orint\'es, Rev. Maghr\'ebine Math. 1(1) (1992), 79--92.

\bibitem{LS2008} E. Loz, J. \v Sir\'a\v n, \textit{New record graphs in the degree-diameter problem}, Australas. J. Combin., \textbf{41} (2008), 63--80

\bibitem{wiki} 
\href{http://combinatoricswiki.org/wiki/The_Degree_Diameter_Problem_for_Cayley_Graphs}%
{combinatoricswiki.org/wiki/The\_Degree\_Diameter\_Problem\_for\_Cayley\_Graphs}



\bibitem{AGS} D. Archdeacon, P. Gvozdjak, J. \v Sir\'a\v n, \textit{Constructing and forbidding automorphisms in lifted maps}, Math. Slovaca 
\textbf{47} (1997), no. 2, 113-129

\bibitem{biggs1} N. Biggs, \textit{Automorphisms of imbedded graphs}, J. Combin. Theory Ser.~B \textbf{11} (1971), 132-138

\bibitem{sing-bo} R. P. Bryant, D. Singerman, \textit{Foundations of the theory of maps
on surfaces with boundary}, Quart. J. Math. Oxford (2) \textbf{36} (1985), 17-41

\bibitem{r=n} M. J. Grannell, T. S. Griggs, J. \v Sir\'a\v n: \textit{Hamiltonian embeddings from triangulations}, Bull. Lond. Math. Soc. 
\textbf{39} (2007),  no. 3, 447-452.

\bibitem{gross-tucker} J. L. Gross, T. W. Tucker, \textit{Topological graph theory}, Wiley, New York (1987)

\bibitem{quadrilat_non-or} N. Hartsfield, G. Ringel, \textit{Minimal quadrangulations of nonorientable surfaces}, J. Combin. Theory Ser.~A 
\textbf{50} (1989), no. 2, 186-195

\bibitem{quadrilat_or} N. Hartsfield, G. Ringel, \textit{Minimal quadrangulations of orientable surfaces}, J. Combin. Theory Ser.~B 
\textbf{46} (1989), no. 1, 84-95

\bibitem{JJ} L. D. James, G. A. Jones, \textit{Regular orientable imbeddings of complete graphs}, J. Combin. Theory Ser.~B 
\textbf{39} (1985), 353-367

\bibitem{RSJTW} R. B. Richter, J. \v Sir\'a\v n, R. Jajcay, T. W. Tucker, M. E. Watkins, \textit{Cayley maps}, J. Combin. Theory Ser.~B 
\textbf{95} (2005) 189-245

\bibitem{ringel} G. Ringel, \textit{Map Color Theorem}, Springer (1974)

\bibitem{tucker} T. W. Tucker, \textit{Symmetric embeddings of Cayley graphs in nonorientable surfaces}, Graph Theory, 
Combinatorics and Applications (Kalamazoo, 1988), 1105-1120, Wiley-Interscience, New York (1991)
\fi
%%%%%%%%%%%%%%%%%%%%%%%%%%%%%%%%%%%%%%%%%%%%%%%%%%%%%%%%%%%%%%%%%%%%%%%%%%%%%%%%%%%%
\end{thebibliography}
\end{document}